\documentclass[12pt]{amsart}
\usepackage[pdftex]{graphicx}
\usepackage{amsmath,amssymb,stmaryrd}
\usepackage{amsthm}
\usepackage{enumerate}
\usepackage{comment}
\usepackage[all]{xy}
\usepackage[pdftex]{color}
\usepackage{hyperref}
\usepackage[margin=30truemm]{geometry}

\theoremstyle{theorem}
\newtheorem{thm}{Theorem}[section]
\theoremstyle{definition}
\newtheorem{dfn}[thm]{Definition}
\newtheorem{rem}[thm]{Remark}
\theoremstyle{theorem}
\newtheorem{lem}[thm]{Lemma}
\newtheorem{prop}[thm]{Proposition}
\newtheorem{cor}[thm]{Corollary}

\numberwithin{equation}{section}

\graphicspath{{./fig/}}

\title{Some properties of grid homology for MOY graphs}
\author{Hajime Kubota}
\date{\today}

\begin{document}
\maketitle
\begin{abstract}
Grid homology for MOY graphs is immediately defined from grid homology for transverse spatial graphs developed by Harvey and O'Donnol \cite{spatial}.
We studied some properties of grid homology for MOY graphs such as the oriented skein relation, the effect of contracting an edge, and the effect of uniting parallel edges with the same orientation.
\end{abstract}

\section{Introduction}

Grid homology is the combinatorial link Floer homology developed by Manolescu, Ozsv\'{a}th, Szab\'{o}, Thurston \cite{oncombinatorial}.
Knot/link Floer homology is widely studied because it has various applications to low-dimensional topology.
For example, knot Floer homology detects the genus of a knot \cite{KFH-genus}.
Knot Floer homology is called the categorification of the Alexander polynomial since the graded Euler characteristic coincides with the Alexander polynomial.

Harvey and O'Donnol extended grid homology to transverse spatial graphs \cite{spatial}.
Roughly, a transverse spatial graph is a spatial graph with a separation for incoming edges and outgoing edges at each vertex.
For a transverse spatial graph $f\colon G\to S^3$, Harvey and O'Donnol defined the extended grid homology $HFG^-(f)$ as a relatively bigraded $\mathbb{F}[U_1,\dots, U_V]$-modules and $\widehat{HFG}(f)$ as a relatively bigraded $\mathbb{F}$-vector space, where $V$ is the number of vertices of $f$ and $\mathbb{F}=\mathbb{Z}/2\mathbb{Z}$.
A $relatively\ bigraded$ means that a module or vector space with absolute $\mathbb{Z}$-valued Maslov grading and relative $H_1(S^3-f(G))$-valued Alexander grading.
They  showed that the hat version of their grid homology is the combinatorial version of sutured Floer homology \cite[Theorem 6.6]{spatial}.
As a corollary, the graded Euler characteristic of their hat version coincides with the torsion invariant $T(E(f),\gamma(f))\in\mathbb{Z}[H_1(E(f))]$ of Friedl, Juh\'{a}sz, and Rasmussen \cite{decategorification-of-sutured-Floer}, where $(E(f),\gamma(f))$ is the sutured manifold determined by $f$.
Harvey and O'Donnol defined the Alexander polynomial for transverse spatial graphs as the torsion invariant $T(E(f),\gamma(f))$.

Harvey and O'Donnol also mentioned that their grid homology is the same as Heegaard Floer homology for some balanced bipartite graphs defined by Bao \cite{HF-bipartite} if the graphs have no sinks, sources, or cut edges.
For a transverse spatial graph $f$, we can take the balanced bipartite spatial graph $f'$ uniquely.
Because $f'$ is also a transverse spatial graph, we can take a (good) grid diagram $g'$ representing $f'$.
Bao pointed out that if we regard $g'$ as a special Heegaard diagram, her chain complex coincides with the grid chain complex of $g'$.

Roughly, transverse spatial graphs with a certain edge coloring which is called balanced coloring are called MOY graphs.
Combining balanced coloring with the extended grid homology for transverse spatial graphs, we can quickly define grid homology for MOY graphs, we denoted by $HF^-(f,\omega)$ or $\widehat{HF}(f,\omega)$.
The Alexander  grading of our grid homology takes values in $\mathbb{Z}$ because we can take the homomorphism $\omega\colon H_1(S^3-f(G))\to\mathbb{Z}$ determined by the balanced coloring of $f$.
We observed grid homology for MOY graphs and its graded Euler characteristic $\Delta_{(f,\omega)}(t)\in\mathbb{Z}[t^{\pm1}]$.

Our main goal is to study some properties of grid homology for MOY graphs such as the oriented skein relation, the effect of contracting an edge, and the effect of uniting parallel edges with the same orientation.

\subsection{MOY graphs}
An oriented spatial graph $f$ is the embedding image of an abstract directed graph in $S^3$.
A \textbf{transverse spatial graph} is an oriented spatial graph such that there is a small disk $D\subset S^3$ that separates the incoming edges and the outgoing edges for each vertex (Figure \ref{fig:transverse}).
In this paper, we always assume that every transverse spatial graph is sourceless and sinkless, which means that each vertex has both incoming edges and outgoing edges.

\begin{dfn}
\label{dfn:balanced}
We denote the set of edges of $f$ by $E(f)$ and the set of vertices of $f$ by $V(f)$.
For a vertex $v$, let $\mathrm{In}(v)$ be the set of edges incoming to $v$ and $\mathrm{Out}(v)$ be the set of edges outgoing to $v$.
\begin{enumerate}
    \item A \textbf{balanced coloring} $\omega$ for $f$ is a map $E(f)\to\mathbb{Z}$ satisfying $\sum_{e\in \mathrm{In}(v)}\omega(e)=\sum_{e\in \mathrm{Out}(v)}\omega(e)$ for each $v\in V(f)$
    \item An \textbf{MOY graph} $(f,\omega)$ is a transverse spatial graph with a balanced coloring.
\end{enumerate}

\end{dfn}

\begin{rem}
Mellor, Kong, Lewald, and Pigrish \cite{coloring} define a balanced spatial graph as a pair of a spatial graph and a balanced coloring, which is an MOY graph without the transverse condition.
It is confusing that Vance\cite{grid-tau} and the author paper\cite{grid-upsilon-concordance} defined a balanced spatial graph, which is a transverse spatial graph satisfying that the number of incoming edges equals the number of outgoing edges at each vertex.
Note that a balanced spatial graph of Vance and the author is a special case of an MOY graph.
\end{rem}

\subsection{Main results}
Let $W(i)$ be a two-dimensional graded vector space $W(i)\cong\mathbb{F}_{0,0}\oplus\mathbb{F}_{-1,-i}$, where $\mathbb{F}=\mathbb{Z}/2\mathbb{Z}$.
For a bigraded $\mathbb{F}$-vector space $X$, $X\llbracket a,b\rrbracket$ denotes a \textbf{shift} of a bigraded $\mathbb{F}$-vector space such that $X\llbracket a,b\rrbracket_{d,s}=X_{d+a,s+b}$. Then,
\[
X\otimes W(i)\cong X\oplus X\llbracket1,i\rrbracket.
\]
\begin{figure}
\centering
\includegraphics[scale=0.6]{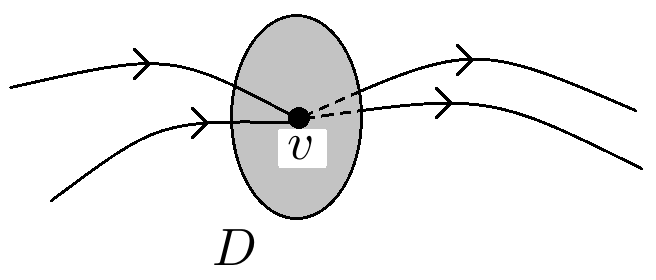}
\caption{A small disk at a vertex of a transverse spatial graph}
\label{fig:transverse}
\end{figure}

\begin{figure}
\centering
\includegraphics[scale=0.35]{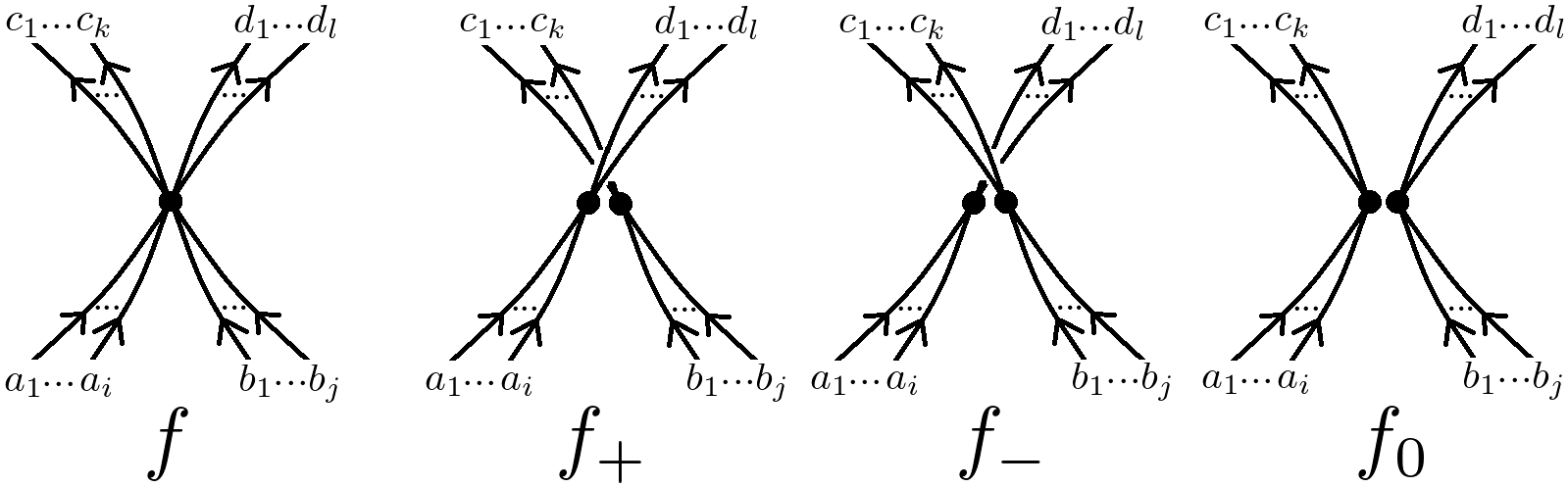}
\caption{An oriented skein triple of MOY graphs. Edges with the same label correspond to each other.}
\label{fig:triple1}
\end{figure}

\begin{dfn}
Let $f,f_+,f_-,f_0$ be four transverse spatial graphs as in Figure \ref{fig:triple1} and $\omega$ be a balanced coloring for $f$.
We call a 4-tuple $(f,f_+,f_0,\omega)$ or $(f,f_-,f_0,\omega)$ an \textbf{oriented skein triple of MOY graphs} if the balanced colorings of $f_+,f_-,f_0$ are naturally determined by $\omega$ and if $\sum_{s=1}^{i}\omega(a_s)=\sum_{s=1}^{j}\omega(b_s)=\sum_{s=1}^{k}\omega(c_s)=\sum_{s=1}^{l}\omega(d_s)$.
\end{dfn}

The following theorem implies the existence of exact sequences of grid homology $\widehat{HF}(f,\omega)$ for oriented skein triple of MOY graphs.
\begin{thm}\rm
\label{thm:main1}
(1) Suppose that $(f,f_+,f_0,\omega)$ is an oriented skein triple of MOY graphs as in Figure \ref{fig:triple1}.
Let $\omega_+,\omega_0$ be balanced colorings for $f_+,f_0$ determined by $\omega$ and let $m=\sum_{s=1}^{i}\omega(a_s)$.
Then, there is an exact triangle as absolute Maslov graded, relative Alexander graded $\mathbb{F}$-vector spaces;
\[
\xymatrix{
\widehat{HF}(f_0,\omega_0)\ar[rr]^-{(-1,0)}&\ar@{}@<0.8ex>[d]&\widehat{HF}(f_+,\omega_+)\ar[dl]^-{(1,0)}\\
&\widehat{HF}(f,\omega)\otimes W(m)\ar[ul]^-{(-1,0)}&
}
\]
(2) Suppose that $(f,f_-,f_0,\omega)$ is an oriented skein triple of MOY graphs as in Figure \ref{fig:triple1}.
Let $\omega_-,\omega_0$ be balanced colorings for $f_-,f_0$ determined by $\omega$ and let $m=\sum_{s=1}^{i}\omega(a_s)$.
Then, there is an exact triangle as absolute Maslov graded, relative Alexander graded $\mathbb{F}$-vector spaces;
\[
\xymatrix{
\widehat{HF}(f_-,\omega_-)\ar[rr]^-{(-1,0)}&\ar@{}@<0.8ex>[d]&\widehat{HF}(f_0,\omega_0)\ar[dl]^-{(1,0)}\\
&\widehat{HF}(f,\omega)\otimes W(m)\ar[ul]^-{(-1,0)}&
}
\]
\end{thm}
\begin{dfn}
For a graph grid diagram $g$, let $\Delta_g(t)$ be the graded Euler characteristic
\[
\Delta_g(t)=\sum_{d,s\in\mathbb{Z}}(-1)^d\mathrm{dim}_\mathbb{F}\widehat{HF}(g,\omega)_{d,s}\cdot t^s.
\]
\end{dfn}
This is an invariant for transverse spatial graphs up to multiplication by $\pm t^{\pm n}$ because $\widehat{HF}(g)$ is independent of the choice of $g$ representing $f$, up to shift of Alexander grading.
We denote by $\Delta_{(f,\omega)}(t)$.

\begin{cor}
(1) Let $(f,f_+,f_0,\omega)$ be an oriented skein triple of MOY graphs and $m=\sum_{s=1}^{i}\omega(a_s)$.
Then, there is a pair of integers $(m_1,m_2)$ such that
\[
t^{m_1}\cdot\Delta_{(f_0,\omega_0)}(t)-t^{m_2}\cdot\Delta_{(f_+,\omega_+)}(t)\doteq(1-t^m)\cdot\Delta_{(f,\omega)}(t),
\]
where $\doteq$ means equal up to multiplication by $\pm t^{\pm n}$.

(2) Let $(f,f_-,f_0,\omega)$ be an oriented skein triple of MOY graphs and $m=\sum_{s=1}^{i}\omega(a_s)$.
Then, there is a pair of integers $(m_1,m_2)$ such that
\[
t^{m_1}\cdot\Delta_{(f_0,\omega_0)}(t)-t^{m_2}\cdot\Delta_{(f_-,\omega_-)}(t)\doteq(1-t^m)\cdot\Delta_{(f,\omega)}(t),
\]
\end{cor}
\begin{rem}
We need a pair of integers $(m_1,m_2)$ because $\Delta_{(f,\omega)}(t)$ is defined up to multiplication by $\pm t^m$, equivalently, $\widehat{HF}(f,\omega)$ is an invariant as relativity Alexander graded vector space.
\end{rem}

\begin{figure}
\centering
\includegraphics[scale=0.7]{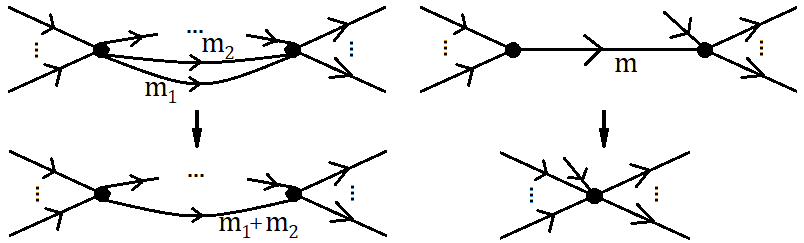}
\caption{Left: graphs in Theorem \ref{thm:main2}. Right: contracting an edge of Theorem \ref{thm:main3}}
\label{fig:para-cont}
\end{figure}

\begin{thm}
\label{thm:main2}
Let $(f,\omega)$ be an MOY graph.
Suppose that $f$ has two parallel edges $e_1,e_2$ with the same orientation.
Let $(f',\omega')$ be an MOY graph obtained by replacing $e_1,e_2$ with one edge whose weight is $\omega(e_1)+\omega(e_2)$ and the orientation of the new edge is the same as the orientation of $e_1,e_2$
(See the left part of Figure \ref{fig:para-cont}).
Let $V$ be the number of vertices of $f,f'$.
Then,
\begin{align*}
HF^-(f,\omega) &\cong HF^-(f',\omega'),\\
\widehat{HF}(f,\omega) &\cong\widehat{HF}(f',\omega'),
\end{align*}
as absolute Maslov graded, relative Alexander graded $\mathbb{F}[U_1,\dots,U_V]$-modules and as absolute Maslov graded, relative Alexander graded $\mathbb{F}$-vector spaces, respectively.
\end{thm}

\begin{thm}
\label{thm:main3}
Let $(f,\omega)$ be an MOY graph.
Let $e$ be an edge of $f$ which is not a loop.
Suppose that at either of its endpoints, $e$ is the only outgoing edge or incoming edge.
Let $(f',\omega')$ be an MOY graph such that $f'$ is obtained from $f$ by contracting $e$ and $\omega'$ is naturally determined by $\omega$
(See the right part of Figure \ref{fig:para-cont}).
Let $V$ be the number of vertices of $f'$.
Then,
\begin{align*}
HF^-(f,\omega) &\cong HF^-(f',\omega'),\\
\widehat{HF}(f,\omega) &\cong\widehat{HF}(f',\omega')\otimes W(\omega(e)).
\end{align*}
as absolute Maslov graded, relative Alexander graded $\mathbb{F}[U_1,\dots,U_V]$-modules and as absolute Maslov graded, relative Alexander graded $\mathbb{F}$-vector spaces, respectively.
\end{thm}

\section{Grid homology for MOY graphs}
\label{sec:background}

\subsection{Background in bigraded chain complexes}
\begin{dfn}
Let $\mathbb{F}=\mathbb{Z}/2\mathbb{Z}$.
\begin{itemize}
    \item A \textbf{bigraded $\mathbb{F}[U_1,\dots,U_n]$-module} $M$ is an $\mathbb{F}[U_1,\dots,U_n]$-module which has a splitting $M=\bigoplus_{d,s\in \mathbb{Z}}M_{d,s}$ as a $\mathbb{F}$-module so that $V_i$ sends $M_{d,s}$ to $M_{d-2,s-\omega(i)}$ for all $i=1,\dots,n$, where $\omega(i)$'s are some integers.
    \item A \textbf{bigraded $\mathbb{F}[U_1,\dots,U_n]$-module homomorphism} $f\colon M\rightarrow M'$ is a homomorphism between two bigraded $\mathbb{F}[U_1,\dots,U_n]$-modules which sends $M_{d,s}$ to $M'_{d,s}$ for all $d,s\in\mathbb{Z}$.
    \item $\mathbb{F}[U_1,\dots,U_n]$-module homomorphism $f\colon M\rightarrow M'$ is a \textbf{homogeneous of degree $(m,t)$} if it sends $M_{d,s}$ to $M'_{d+m,s+t}$ for all $d,s\in\mathbb{Z}$.
    \item $(C,\partial)$ is a \textbf{bigraded chain complex over} $\mathbb{F}[U_1,\dots,U_n]$ if $C$ is bigraded $\mathbb{F}[U_1,\dots,U_n]$-module and if $\mathbb{F}[U_1,\dots,U_n]$-module homomorphism $\partial\colon C\rightarrow C$ satisfies $\partial\circ\partial=0$ and is homogeneous of degree $(-1,0)$.
\end{itemize}
\end{dfn}
Throughout this paper, we often consider chain maps, chain homotopy equivalences, etc. as bigraded $\mathbb{F}[U_1,\dots, U_n]$-module homomorphisms.

\subsection{The definition of grid homology for MOY graphs}
This section provides an overview of grid homology for MOY graphs.
Roughly, the Alexander grading of our grid homology is obtained by taking homomorphism $\omega\colon H_1(S^3-f(G))\to\mathbb{Z}$ determined by the balanced coloring and collapsing $H_1(S^3-f(G))$-valued Alexander grading of Harvey and O'Donnol into $\mathbb{Z}$.
So the almost definitions are the same as those of Harvey and O'Donnol.

A \textbf{planar graph grid diagram} $g$ is a $n\times n$ grid of squares some of that are decorated with an $X$- or $O$- (sometimes $O{}^*$-) markings with the following conditions.
\begin{enumerate}[(i)]
\item There is just one $O$ or $O{}^*$ on each row and column.
\item There is at least one $X$ on each row and column.
\item $O$'s (or $O{}^*$'s) and $X$'s do not share the same square.
\end{enumerate}
We denote the set of $O$-markings and $O^*$-markings by $\mathbb{O}$ and the set of $X$-markings by $\mathbb{X}$ and use the labeling of markings as $\{O_i\}_{i=1}^n$ and $\{X_j\}_{j=1}^m$.

Graph grid diagrams realize spatial graphs by connecting the $O$- (or $O^*$)- markings and $X$-markings by horizontal or vertical segments in each row and column, and assuming that the vertical segments always cross above the horizontal segments.
$O^*$-markings correspond to vertices, $O$- and $X$-markings to interior of edges of a transverse spatial graph.

Harvey and O'Donnol showed that any two graph grid diagrams representing the same spatial graph are connected by a finite sequence of the graph grid moves.
The graph grid moves are the following three moves$\colon$
\begin{itemize}
\item \textbf{Cyclic permutation} (Figure \ref{fig:cyccomm}) permuting the rows or columns cyclically.
\item \textbf{Commutation'} (Figure \ref{fig:cyccomm}) permuting two adjacent columns satisfying the following condition; vertical line segments $\textrm{LS}_1,\textrm{LS}_2$ on the torus such that (1) $\mathrm{LS}_1\cup\mathrm{LS}_2$ contains all the $X$'s and $O$'s in the two adjacent columns, (2) the projection of $\mathrm{LS}_1\cup\mathrm{LS}_2$ to a single vertical circle $\beta_i$ is $\beta_i$
, and (3) the projection of their endpoints, $\partial(\mathrm{LS}_1)\cup\partial(\mathrm{LS}_2)$, to a single $\beta_i$ is precisely two points. Permuting two rows is defined in the same way.
\item \textbf{(de-)stabilization'} (Figure \ref{fig:sta'}) let $g$ be an $n\times n$ graph grid diagram and choose an $X$-marking. Then $g'$ is called a stabilization' of $g$ if it is an $(n+1)\times(n+1)$ graph grid diagram obtained by adding a new row and column next to the $X$-marking of $g$, moving the $X$-marking to next column, and putting new one $O$-marking just above the $X$-marking and one $X$-marking just upper left of the $X$-marking. The inverse of stabilization is called destabilization.
\end{itemize}

\begin{figure}
\centering
\includegraphics[scale=0.5]{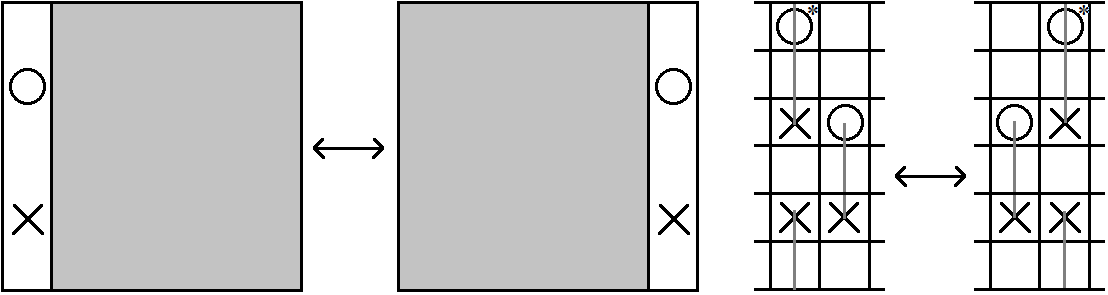}
\caption{Cyclic permutation and commutation', gray lines are $\mathrm{LS}_1$ and $\mathrm{LS}_2$}
\label{fig:cyccomm}
\includegraphics[scale=0.5]{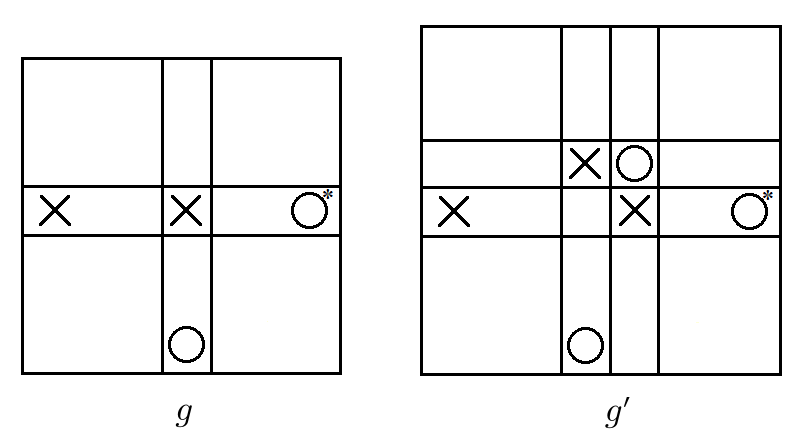}
\caption{Stabilization'}
\label{fig:sta'}

\includegraphics[scale=0.6]{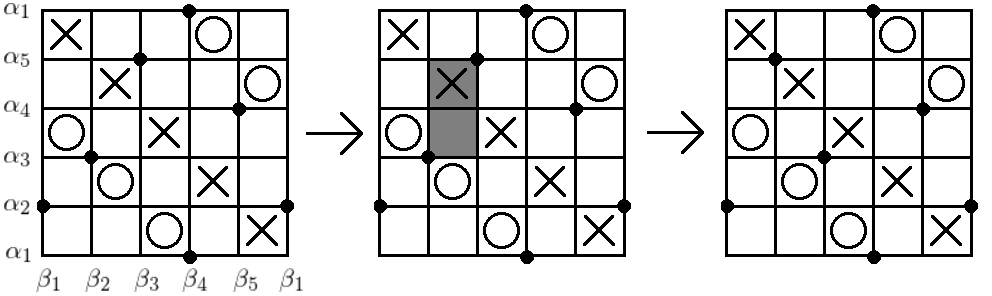}
\caption{An example of a state and a rectangle}
\label{fig:state}
\end{figure}

A \textbf{toroidal graph grid diagram} is a graph grid diagram that we think it as a diagram on the torus obtained by identifying edges in a natural way.
We write the horizontal and vertical circles, which separate squares as $\boldsymbol{\alpha}=\{\alpha_i\}_{i=1}^n$ and $\boldsymbol{\beta}=\{\beta_j\}_{j=1}^n$.
We assume that every toroidal diagram is oriented naturally.

A \textbf{state} $\mathbf{x}$ of $g$ is a bijection $\boldsymbol{\alpha}\rightarrow\boldsymbol{\beta}$.
We denote by $\mathbf{S}(g)$ the set of states of $g$.
We describe a state as $n$ points on the graph grid diagram (Figure \ref{fig:state}).

For $\mathbf{x,y}\in S$(g), a \textbf{domain} $\psi$ from $\mathbf{x}$ to $\mathbf{y}$ is a formal sum of the closure of squares, which is satisfying the following conditions;
\begin{itemize}
    \item $\psi$ is divided by $\boldsymbol{\alpha}\cup\boldsymbol{\beta}$
    \item $\partial(\partial_\alpha\psi)=\mathbf{y}-\mathbf{x}$ and $\partial(\partial_\beta\psi)=\mathbf{x}-\mathbf{y}$, where $\partial_\alpha\psi$ is the portion of the boundary of $\psi$ in the horizontal circles $\alpha_1\cup\dots\cup\alpha_n$ and $\partial_\beta\psi$ is the portion of the boundary of $\psi$ in the vertical ones.
\end{itemize}
A domain $\psi$ is \textbf{positive} if the coefficient of any square is nonnegative.
Here, we always assume that any domain is positive.
Let $\pi(\mathbf{x,y})$ denote the set of positive domains from $\mathbf{x}$ to $\mathbf{y}$.

Consider $\mathbf{x,y\in S}(g)$ satisfying that $|\mathbf{x\cap y}|=n-2$.
An \textbf{rectangle} $r$ from $\mathbf{x}$ to $\mathbf{y}$ is a domain that satisfies that $\partial r$ is the union of four segments.
A rectangle $r$ is \textbf{empty} if $\mathbf{x}\cap\mathrm{Int}(r)=\mathbf{y}\cap\mathrm{Int}(r)=\phi$.
Let $\mathrm{Rect}^\circ(\mathbf{x,y})$ be the set of empty rectangles from $\mathbf{x}$ to $\mathbf{y}$.

\begin{dfn}
For a graph grid diagram $g$ representing $f$ with balanced coloring $\omega$, a \textbf{weight} $\omega_g\colon\mathbb{O}\cup\mathbb{X}\to\mathbb{Z}$ is a map naturally determined by $\omega$ as follows;
\begin{itemize}
    \item $\omega_g(O_i)=\omega(e)$ if $O_i$ corresponds to the interior of the edge $e$.
    \item $\omega_g(X_j)=\omega(e)$ if $X_j$ corresponds to the interior of the edge $e$.
    \item $\omega_g(O_i)=\sum_{e\in\mathrm{In}(v)}\omega(e)=\sum_{e\in\mathrm{Out}(v)}\omega(e)$ if $O_i$ corresponds to the vertex $v$, where $\mathrm{In}(v)$ is the set of edges that go into $v$ and $\mathrm{Out}(v)$ is the set of edges that go out of $v$.
\end{itemize}
We abbreviate $\omega_g$ to $\omega$ as long as there is no confusion.
\end{dfn}

The grid chain complex $(CF^-(g,\omega),\partial^-)$ is a module over $\mathbb{F}[U_1,\dots,U_n]$ freely generated by $\mathbf{S}(g)$ and the $U_i$'s are formal variables corresponding to the $O_i$'s in $g$.
The differential $\partial^-$ is defined as counting empty rectangles by
\[
\partial^-(\mathbf{x})=\sum_{\mathbf{y}\in\mathbf{S}(g)}\left(
\sum_{\{r\in \mathrm{Rect}^\circ(\mathbf{x,y})|r\cap\mathbb{X}=\phi\}}
U_1^{O_1(r)}\cdots U_n^{O_n(r)}
\right)\mathbf{y},
\]
where $O_i(r)=1$ if $r$ contains $O_i$ and $O_i(r)=0$ otherwise for $i=1,\dots,n$. 

There are two gradings for $CF^-(g)$, the \textbf{Maslov grading} and the \textbf{Alexander grading}.
A planar realization of toroidal diagram $g$ is a planar figure obtained by cutting the toroidal diagram $g$ along $\alpha_i$ and $\beta_j$ for some $i$ and $j$, and putting on $[0,n)\times[0,n)\in\mathbb{R}^2$ in a natural way.
For two points $(a_1,a_2),(b_1,b_2)\subset\mathbb{R}^2$, let $(a_1,a_2)<(b_1,b_2)$ if $a_1<b_1$ and $a_2<b_2$.
For two sets of finitely points $A,B\subset\mathbb{R}^2$, let $\mathcal{I}(A,B)$ be the number of pairs $a\in A,b\in B$ with $a<b$ and let $\mathcal{J}(A,B)=(\mathcal{I}(A,B)+\mathcal{I}(B,A))/2$.
\begin{dfn}
For $\mathbf{x\in S}(g)$, the Maslov grading $M(\mathbf{x})$ and the Alexander grading $A(\mathbf{x})$ are defined by
\begin{align}
M(\mathbf{x})&=\mathcal{J}(\mathbf{x}-\mathbb{O},\mathbf{x}-\mathbb{O})+1,
\label{mm}
\\
A(\mathbf{x})&=\mathcal{J}(\mathbf{x},\sum_{j=1}^m\omega(X_j)\cdot X_j-\sum_{i=1}^n\omega(O_i)\cdot O_i).
\label{aa}
\end{align}
These two gradings are extended to the whole of $CF^-(g,\omega)$ by
\begin{align}
\label{uu}
M(U_i)=-2,\ A(U_i)=-\omega(O_i)\ (i=1,\dots,n).
\end{align}
\end{dfn}
Alexander grading is defined as counting pairs with multiplying weights of markings.
\begin{rem}
Our Alexander grading is obtained by taking the homomorphism $H_1(S^3-f(G))\to\mathbb{Z}$, which sends the meridian of each edge to the assigned integer determined by $\omega$, and the rest of the definitions are the same as the definitions of Harvey and O'Donnol.
So the invariance of our grid homology is naturally inherited from the arguments of Harvey and O'Donnol.
\end{rem}
Also, note that the Alexander grading is not well-defined as a toroidal diagram, however, relative Alexander grading $A^{rel}(\mathbf{x,y})=A(\mathbf{x})-A(\mathbf{y})$ is well-defined.
It is easy to verify that the differential $\partial^-$ drops Maslov grading by one and preserves Alexander grading, using the same argument such as Proposition 4.18\cite{spatial} and Proposition 3.7,\cite{grid-tau}.
So we can get this:
\begin{prop}
$(CF^-(g,\omega),\partial^-)$ is an absolute Maslov graded, relative Alexander graded chain complex.
\end{prop}

Suppose that $O$-markings are labeled so that $O_1,\dots,O_V$ are $O^*$-markings and $O_{V+1},\dots,O_n$ are $O$-markings.
Let $\mathcal{U}$ be the minimal subcomplex of $CF^-(g,\omega)$ containing $U_1CF^-(g,\omega)\cup\dots\cup U_VCF^-(g,\omega)$.
$(\widehat{CF}(g,\omega),\widehat{\partial})$ is an absolute Maslov graded, relative Alexander graded chain complex over $\mathbb{F}$-vector space obtained by letting $\widehat{CF}(g,\omega)=CF^-(g,\omega)/\mathcal{U}$ and $\widehat{\partial}$ be the map induced by $\partial^-$.
\begin{dfn}
$HF^-(g,\omega)$ and $\widehat{HF}(g,\omega)$ are the homology of $(CF^-(g,\omega),\partial^-)$ and $(\widehat{CF}(g,\omega),\widehat{\partial})$, respectively.
\end{dfn}
Note that $HF^-(g,\omega)$ is regarded as an absolute Maslov graded, relative Alexander graded $\mathbb{F}[U_1,\dots,U_V]$-module, and $\widehat{HF}(g,\omega)$ is an absolute Maslov graded and relative Alexander graded $\mathbb{F}$-vector space.

\begin{prop}
$HF^-(g,\omega)$ and $\widehat{HF}(g,\omega)$ are invariants for transverse balanced-colored spatial graphs.
\end{prop}
\begin{proof}
Our chain complexes are obtained from the chain complexes of Harvey and O'Donnol by taking some homomorphism $H_1(S^3-f(G))\to\mathbb{Z}$.
The only difference between these two versions are domains of Alexander grading.
So the invariances are inherited.
\end{proof}

\begin{rem}
$\widehat{CF}(g,\omega)$ can be regarded as absolute Maslov graded, relative Alexander graded $\mathbb{F}$-vector space with basis $\{U_{V+1}^{k_{V+1}}\cdots U_n^{k_n}\cdot\mathbf{x}|k_i\geq0,\mathbf{x\in S}(g)\}$ and the differential
\[
\widehat{\partial}(\mathbf{x})=\sum_{\mathbf{y}\in\mathbf{S}(g)}\left(
\sum_{\{r\in \mathrm{Rect}^\circ(\mathbf{x,y})|r\cap\mathbb{X}=r\cap\mathbb{O}^*=\phi\}}
U_{V+1}^{O_{V+1}(r)}\cdots U_n^{O_n(r)}
\right)\mathbf{y},
\]
where $\mathbb{O}^*$ is the set of $O^*$-markings.
\end{rem}

\section{The proofs for the main Theorem \ref{thm:main1}}
\begin{lem}[Lemma A.3.10,\cite{grid-book}]
\label{lem:cone}
Let $C$, $C'$, and $C''$ be three bigraded chain complexes, and $f\colon C\to C'$ and $g\colon C'\to C''$ are chain maps that are homogeneous of degrees $(a,p)$ and $(b,q)$ respectively.
Then there is a chain map $\Phi\colon\mathrm{Cone}(g)\to \mathrm{Cone}(f)$ which is homogeneous of degree $(-b-1,-q)$ and whose induced map on homology fits into an exact triangle
\[
\xymatrix{
H(\mathrm{Cone}(g))\ar[rr]^-{(-b-1,-q)}&\ar@{}@<0.8ex>[d]&H(\mathrm{Cone}(g))\ar[dl]^-{(b,q)}\\
&H(\mathrm{Cone}(g\circ f))\ar[ul]^-{(0,0)}&
}
\]
where the pair of integers indicate shifts of bigradings.
\end{lem}
Let $(f,f_+,f_0,\omega)$ be an oriented skein triple of MOY graphs.
Choose graph grid diagrams $(g,g_+,g_0)$, take a point $c_+$ in $g_+$ and $c_0$ in $g_0$, and label the $X$-markings so that each $X$-marking corresponds to the weight of each edge in Figure \ref{fig:triple1}.
Let $m=\sum_{n=1}^i\omega_X(a_n)=\sum_{n=1}^j\omega_X(j_n)=\sum_{n=1}^k\omega_X(c_n)=\sum_{n=1}^l\omega_X(d_n)$.
Here $\omega_+,\omega_0$ denotes the weights determined by $\omega$.
\begin{figure}
\centering
\includegraphics[scale=0.47]{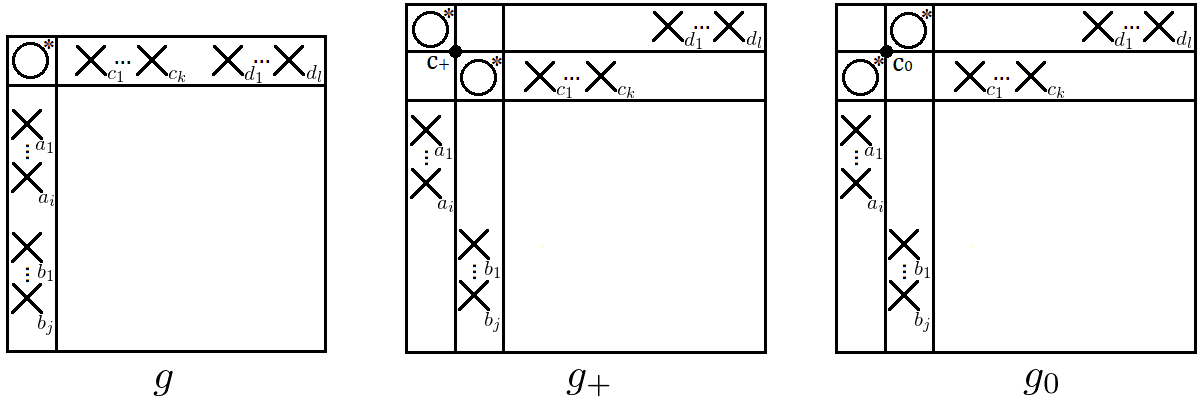}
\caption{Graph grid diagrams representing $f$,$f_+$, and $f_0$}
\label{fig:triple1g}
\end{figure}

We will decompose the set of states $\mathbf{S}(g_+)$ as the disjoint union $\mathbf{I}(g_+)\cup\mathbf{N}(g_+)$, where $\mathbf{I}(g_+)=\{\mathbf{x\in S}(g_+)|c\in\mathbf{x}\}$ and $\mathbf{N}(g_+)=\{\mathbf{x\in S}(g_+)|c\notin\mathbf{x}\}$.
Let $I_+$ and $N_+$ be a subspace of $\widehat{CF}(g_+,\omega_+)$ as
\begin{align*}
I_+ &=\{U_{V+1}^{k_{V+1}}\cdots U_n^{k_n}\mathbf{x}\in\widehat{CF}(g_+)|\mathbf{x}\in I(g_+)\},\\
N_+ &=\{U_{V+1}^{k_{V+1}}\cdots U_n^{k_n}\mathbf{x}\in\widehat{CF}(g_+)|\mathbf{x}\in N(g_+)\}
\end{align*}
This decomposition gives a decomposition of $\widehat{CF}(g_+,\omega_+)=I_+\oplus N_+$ as a vector space.
Then we can write the differential on $\widehat{CF}(g_+,\omega_+)$ as
\[
\widehat{\partial}=
\begin{pmatrix}
\partial_I^I & 0\\
\partial_I^N & \partial_N^N
\end{pmatrix},
\]
where, for instance, $\partial_I^N$ is a part of the differential such that it is a map from $I_+$ to $N_+$.
So $(N_+,\partial_N^N)$ is a subcomplex of $(\widehat{CF}(g_+,\omega_+),\widehat{\partial})$, and $(I_+,\partial_I^I)$ is a quotient complex.
By the definition of algebraic cone, we see that $\widehat{CF}(g_+,\omega_+)\cong\mathrm{Cone}(\partial_I^N\colon (I_+,\partial_I^I)\to(N_+,\partial_N^N))$.

We can define $I_0$ and $N_0$ for $\widehat{CF}(g_0)$ in the same way.
In this case, the differential can be written as
\[
\begin{pmatrix}
\partial_I^I & \partial_N^I\\
0 & \partial_N^N
\end{pmatrix},
\]
and $\widehat{CF}(g_0,\omega_0)\cong\mathrm{Cone}(\partial_N^I\colon (I_0,\partial_I^I)\to(N_0,\partial_N^N))$.

Consider the bijection $c_+\colon\mathbf{S}(g)\to \mathbf{I}(g_+)$ defined as $\mathbf{x}\mapsto\mathbf{x}\cup\{c_+\}$.
Using \ref{mm} and \ref{aa}, simple calculations show that $c_+$ preserves Maslov grading and increases Alexander grading by $m$.
Also Consider the bijection $c_0\colon\mathbf{I}(g_0)\to \mathbf{S}(g)$ defined as $\mathbf{x}\cup\{c_0\}\mapsto\mathbf{x}$.
This map increases Maslov grading by 1 and preserves Alexander grading.

The followings are easy.
\begin{lem}[Lemma 5.2.5,\cite{grid-book}]
\label{lem:i}
\begin{enumerate}
\item[\rm{(1)}] $c_+$ induces an isomorphism $\widehat{CF}(g,\omega)\cong I_+$ between absolute Maslov graded, relative Alexander graded chain complexes.
\item[\rm{(2)}] $c_0$ induces an isomorphism $I_0\cong \widehat{CF}(g,\omega)$ between absolute Maslov graded, relative Alexander graded chain complexes.
\end{enumerate}
\end{lem}
\begin{proof}
The bijection $c_+$ induces a homomorphism on chain complexes because there is a natural bijection between $\{U_{V+1},\dots,U_n\}$ on $g$ and $\{U_{V+1},\dots,U_n\}$ on $g_+$.
It is easy to check that the induced map by $c_+$ is a chain map.
The induced map is an isomorphism because there is a bijection between rectangles connecting states of $\mathbf{S}(g)$ and states of $\mathbf{I}(g_+)$.

$c_0$ induces an isomorphism in the same way.
\end{proof}
Consider the bijection $\alpha\colon \mathbf{N}(g_+)\to \mathbf{N}(g_0)$ defined as $\mathbf{x}\mapsto\mathbf{x}$.
This map increases Maslov grading by 1 and preserves Alexander grading.
The following lemma can be shown in the same way as Lemma \ref{lem:i}.
\begin{lem}
\label{lem:n}
$\alpha$ induces an isomorphism $N_+\cong N_0$ between absolute Maslov graded, relative Alexander graded chain complexes.
\end{lem}
\begin{proof}[\textbf{proof of theorem \ref{thm:main1}}]
Now we have a diagram such as
\[
\xymatrix{
I_+ \ar[r]^-{\partial_I^N} & N_+ \ar[r]^\alpha & N_0 \ar[r]^-{\partial_N^I} & I_0.
}
\]
Using Lemma \ref{lem:cone}, we can get the following exact triangle:
\[
\xymatrix{
H(\mathrm{Cone}(\partial_N^I))\ar[rr]^-{(0,0)}&\ar@{}@<0.8ex>[d]&H(\mathrm{Cone}(\alpha\circ\partial_I^N))\ar[dl]^-{(-1,0)}\\
&H(\mathrm{Cone}(\partial_N^I\circ \alpha\circ\partial_I^N))\ar[ul]^-{(0,0)}&
}
\]
As mentioned above, $H(\mathrm{Cone}(\partial_N^I))\cong\widehat{HF}(g,\omega)$.
By Lemma \ref{lem:n}, we have $H(\mathrm{Cone}(\alpha\circ\partial_I^N))\cong H(\mathrm{Cone}(\partial_I^N))\llbracket -1,0\rrbracket\cong \widetilde{HF}(g_+,\omega_+)\llbracket -1,0\rrbracket$.

For any state $\mathbf{x}\in I_+$, $(\partial_N^I\circ \alpha\circ\partial_I^N)(\mathbf{x})=0$ because the composite domain $r_1*r_2$ where $r_1$ is appear in $\partial_N^I$ and $r_2$ in $\partial_I^N$ must be an annulus, so some $X$-markings are contained in $r_1$ or $r_2$.

So we have $H(\mathrm{Cone}(\partial_N^I\circ \alpha\circ\partial_I^N))\cong H(\mathrm{Cone}(0\colon I_+\to I_0))$.
The composite map $\partial_N^I\circ \alpha\circ\partial_I^N$ is homogeneous of degree $(-1,0)$.
So using lemma \ref{lem:i}, we see that $H(\mathrm{cone}(0\colon I_+\to I_0))\cong H(I_+)\oplus H(I_0)\cong\widehat{HF}(g,\omega)\llbracket 0,-m\rrbracket\oplus\widehat{HF}(g,\omega)\llbracket 1,0\rrbracket\cong\widehat{HF(g,\omega)}\otimes W(m)\llbracket 0,-m\rrbracket$.
Then we get the following exact triangle:
\[
\xymatrix{
\widehat{HF}(g_0,\omega_0)\ar[rr]^-{(-1,0)}&\ar@{}@<0.8ex>[d]&\widehat{HF}(g_+,\omega_+)\ar[dl]^-{(1,0)}\\
&\widehat{HF}(g,\omega)\otimes W(m)\ar[ul]^-{(-1,0)}&
}
\]
\begin{figure}
\centering
\includegraphics[scale=0.47]{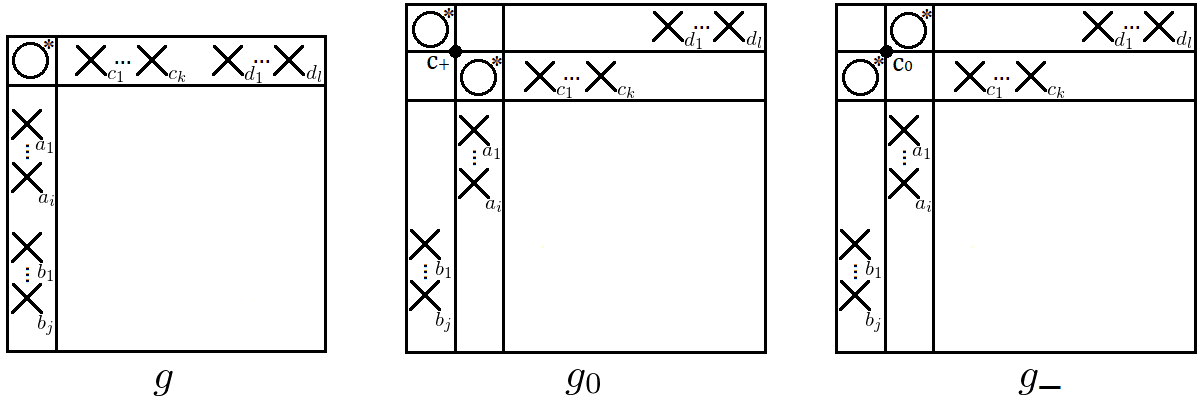}
\caption{Graph grid diagrams representing $f$,$f_0$, and $f_-$}
\label{fig:triple2g}
\end{figure}

Let $(f,f_-,f_0)$ be an oriented skein triple of MOY graphs.
The exact triangle of $(f,f_0,f_-)$ can be obtained in the same way;

First, choose the graph grid diagrams $(g,g_0,g_-)$ as in Figure \ref{fig:triple2g}.

Second, consider bijections $c'_0\colon \mathbf{S}(g)\to I_0$, $c_-\colon I_-\to\mathbf{S}(g)$, and $\beta\colon N_0\to N_-$ in the same way as $c_+$, $c_0$, and $\alpha$, respectively.
$c'_0$ is homogeneous of degree $(0,m)$, $c_-$ is $(1,0)$, and $\beta$ is $(1,-m)$.
These three maps induce isomorphisms on homology.

Finally, the exact triangle is obtained from the following diagram;
\[
\xymatrix{
I_0 \ar[r]^-{\partial_I^N} & N_0 \ar[r]^\beta & N_- \ar[r]^-{\partial_N^I} & I_-.
}
\]
\end{proof}

\section{The proofs for the main Theorem \ref{thm:main2} and Theorem \ref{thm:main3}}
First, we show  Theorem \ref{thm:main2} and Theorem \ref{thm:main3} as corollaries of the following proposition.

\begin{figure}
\centering
\includegraphics[scale=0.5]{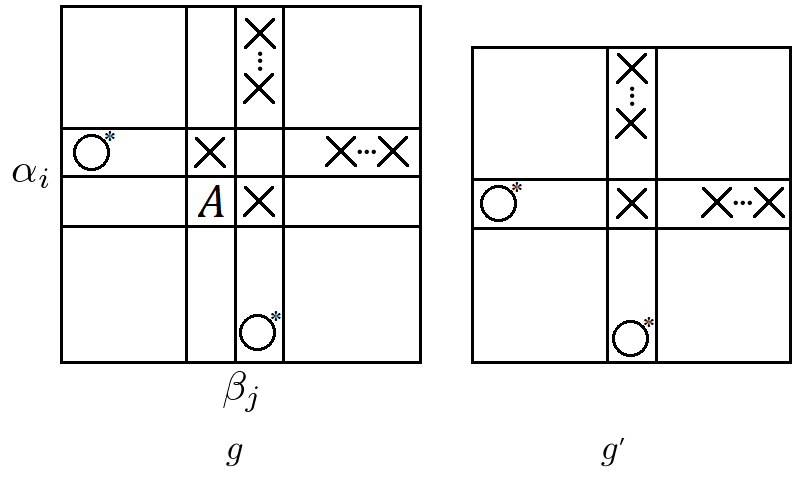}
\caption{Graph grid diagrams in Proposition \ref{prop:desta}}
\label{fig:prop-desta}
\end{figure}

\begin{prop}
\label{prop:desta}
Let $g,g'$ are two graph grid diagrams as in Figure \ref{fig:prop-desta}.
$g'$ is obtained from $g$ by destabilization'-like move.
The weight $\omega'$ for $g'$ is naturally determined by the weight $\omega$ for $g$.
Assume that the following conditions are satisfied:
\begin{itemize}
    \item $g$ is a $(n+1)\times(n+1)$ graph grid diagram.
    \item None of the $X$-markings on the two rows along the horizontal circle $\alpha_i$ line up vertically.
    \item None of the $X$-markings on the two columns along the vertical circle $\beta_j$ line up horizontally.
    \item $A=O$ or $A=O^*$.
    \item there is just one $X$-marking in either row or column containing $A$.
\end{itemize}
Then, there is a quasi-isomorphism between $CF^-(g,\omega)$ and $CF^-(g',\omega')$ as absolute Maslov graded, relative Alexander graded chain complexes over $\mathbb{F}[U_1,\dots, U_n]$,
\end{prop}

\begin{rem}
Since $g$ is a $(n+1)\times(n+1)$ graph grid diagram, $CF^-(g,\omega)$ is defined as an absolute Maslov graded, relative Alexander graded chain complex over $\mathbb{F}[U_1,\dots,U_{n+1}]$, where $U_{n+1}$ are related to the $O$-marking at $A$.
The homology of this chain complex is an invariant for transverse spatial graphs when we regard $HF^-(g,\omega)$ as $\mathbb{F}[U_1,\dots, U_{V}]$-module, where $V$ is the number of vertices of $f$.
\end{rem}

\begin{figure}
\centering
\includegraphics[scale=0.5]{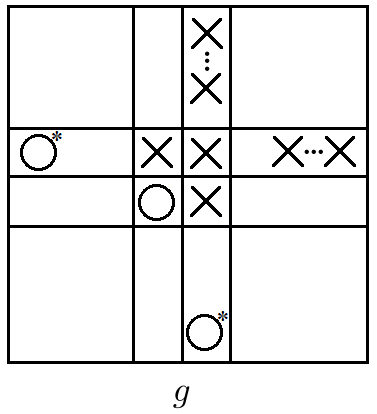}
\caption{Graph grid diagrams in the proof of Theorem \ref{thm:main2}}
\label{fig:parallel2}
\end{figure}

\begin{proof}[\textbf{Proof of theorem \ref{thm:main2}}]
First, we assume that the parallel two edges are connecting two different vertices.
Take a pair of a graph grid diagram and a weight $(g,\omega)$ representing $(f,\omega)$ such that the parallel two edges appear in the $2\times2$ block as in Figure \ref{fig:parallel2}.
Then $g'$ appearing in Figure \ref{fig:prop-desta} represents $f'$ and the weight $\omega'$ for $g'$ naturally determined by $(g,\omega)$ represents the balanced coloring for $f'$.
We get $HF^-(g,\omega) \cong HF^-(g',\omega')$ by using Proposition \ref{prop:desta} as $A=O$ directly.
The quasi-isomorphism $CF^-(g,\omega)\to CF^-(g',\omega')$ induces the quasi-isomorphism $\widehat{CF}(g,\omega)=\frac{CF^-(g,\omega)}{U_1=\dots=U_V=0}\to \widehat{CF}(g',\omega')=\frac{CF^-(g',\omega')}{U_1=\dots=U_V=0}$, so we get $\widehat{HF}(g,\omega) \cong\widehat{HF}(g',\omega')$.

\begin{figure}
\centering
\includegraphics[scale=0.55]{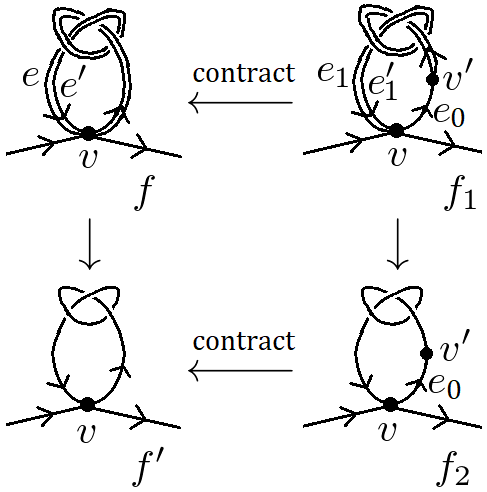}
\caption{Proof of Theorem \ref{thm:main2}, when the parallel two edges are loops. The graphs of the left part is obtained by contracting the edge $e_0$.}
\label{fig:para-single}
\end{figure}

Next, we assume that the two parallel edges $e,e'$ of $f$ are loops at $v$.
Take a new balanced transverse spatial graph $(f_1,\omega_1)$ such that $f$ is obtained from $f_1$ by contracting an edge $e_0$ connecting $v$ and the new vertex $v'$ as in Figure \ref{fig:para-single}, and $\omega_1$ is naturally determined by $\omega$.
By Theorem \ref{thm:main3}, $HF^-(f,\omega)\cong HF^-(f_1,\omega_1)$ as $\mathbb{F}[U_1,\dots,U_V]$-module, where $U_V$ is related to $v$ and $U_{V+1}$ is related to $v'$ (note that $HF^-(f_1)$ is a $\mathbb{F}[U_1,\dots,U_{V+1}]$-module).
Let $(f_2,\omega_2)$ be a balanced transverse spatial graph obtained by unifying the two parallel edges $e_1,e'_1$ determined by $e,e'$.
Since $e_1,e'_1$ are not loops, $HF^-(f_1,\omega_1)\cong HF^-(f_2,\omega_2)$ as $\mathbb{F}[U_1,\dots,U_{V+1}]$-module.
$f'$ is obtained from $f_2$ by contracting $e_0$.
Thus $HF^-(f_2,\omega_2)\cong HF^-(f',\omega')$ as $\mathbb{F}[U_1,\dots,U_V]$-module.
Note that $U_V$ and $U_{V+1}$ are chain homotopic because of Proposition 4.21, \cite{spatial}.
\end{proof}

\begin{proof}[\textbf{Proof of theorem \ref{thm:main3}}]
If $e$ is the only incoming edge of $v$, take a graph grid diagram $g$ representing $f$ such that the $O^*$-marking related to $v$ is located at $A$ as in Figure \ref{fig:prop-desta}.
If $e$ is the only outgoing edge of $v$, take a graph grid diagram $g$ representing $f$ such that the $O^*$-marking related to $v$ is located at $A$ as in Figure \ref{fig:prop-desta}.
We get $HF^-(g,\omega) \cong HF^-(g',\omega')$ by using Proposition \ref{prop:desta} as $A=O^*$ directly.

Next, we show the hat version carefully.
Number the $O^*$-markings in $g'$ as $O_1,\dots, O_V$.
Since $g$ has one more $O^*$-marking, number $O^*$-markings in $g$ as $O_1,\dots, O_V,O_{n+1}$.
By the definition of the hat chain complexes, $\widehat{CF}(g,\omega)=\frac{CF^-(g,\omega)}{U_1=\dots=U_V=U_{n+1}=0}$ and $\widehat{CF}(g',\omega')=\frac{CF^-(g',\omega')}{U_1=\dots=U_V=0}$.
The quasi-isomorphism $CF^-(g,\omega) \to CF^-(g',\omega')$ induces the quasi-isomorphism $\frac{CF^-(g,\omega)}{U_1=\dots=U_V=0}\to \widehat{CF}(g',\omega')$.
Lemma 4.21,\cite{spatial} shows that $U_{n+1}$ is null-homotopic as homogeneous of degree $(-1,-m)$.
Using Lemma 4.31,\cite{spatial}, we see that $\widehat{HF}(g,\omega)\cong \widehat{HF}(g',\omega')\otimes W(m)$
\end{proof}

\begin{proof}[\textbf{Proof of Proposition \ref{prop:desta}}]
The basic idea is the same as the idea of the proof of stabilization invariance (Section 5.2.3,\cite{grid-book}).
The same argument works because we assume that there is just one $X$-marking in either row or column containing $A$.
Take the intersection point $\alpha_i\cap\beta_j$ as $c$.
Label the markings around $c$ as in the right part of Figure \ref{fig:prop-d} and the $O$-marking in the column containing $X_{j+1}$ as $O_n$.
\begin{figure}
\centering
\includegraphics[scale=0.5]{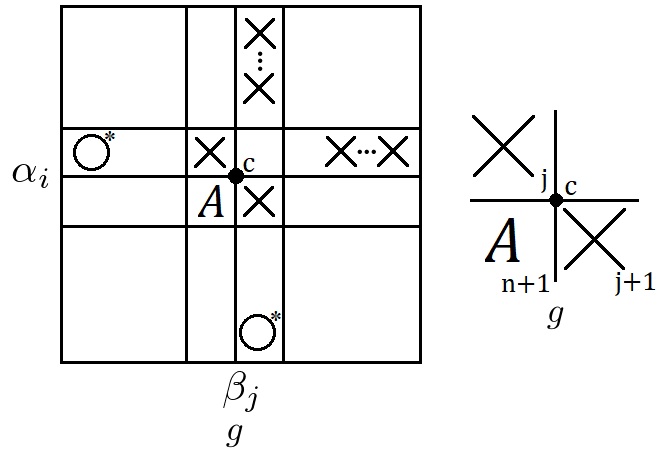}
\caption{The graph grid diagram in the proof of Theorem \ref{thm:main3} and the labels of the markings around $c$}
\label{fig:prop-d}
\end{figure}

Now, we assume that there is just one $X$-marking in the $row$ containing $A$.

First, we introduce some notations like in the previous section.
The set of states $\mathbf{S}(g)$ are decomposed as the disjoint union $\mathbf{I}(g)\cup\mathbf{N}(g)$.
Let $N$ (respectively, $I$) be the submodule of $CF^-(g,\omega)$ generated by $\mathbf{N}(g)$ (respectively, $\mathbf{I}(g)$).
As in the previous section, $N$ is a subcomplex of $CF^-(g,\omega)$.
Using the representation of the differential of $CF^-(g,\omega)$ as a matrix, we see that $CF^-(g,\omega)\cong \mathrm{Cone}(\partial_I^N\colon(I,\partial_I^I)\to(N,\partial_N^N))$.
The bijection between $\mathbf{I}(g)$ and $\mathbf{S}(g')$ induces the isomorphism $c\colon I\to CF^-(g',\omega')[U_{n+1}]$ as $\mathbb{F}[U_1,\dots,U_{n+1}]$-modules, where $CF^-(g',\omega')[U_{n+1}]=CF^-(g',\omega')\otimes_{\mathbb{F}[U_1,\dots,U_n]}\mathbb{F}[U_1,\dots,U_{n+1}]$.
By simple computations, $c$ is homogeneous of degree $(0,-A(\{c\}))$.

Second, we define the homotopy operators.
Consider the map $\mathcal{H}^-_O\colon I\to N$ and $\mathcal{H}^-_X\colon N\to I$ as
\begin{align*}
\mathcal{H}^-_O(\mathbf{x}) &=\sum_{\mathbf{y\in N}(g)}
\sum_{\{r\in \mathrm{Rect}^\circ(\mathbf{x,y})|r\cap\mathbb{X}=\phi,O_{n+1}\in r\}}
U_{1}^{O_{1}(r)}\cdots U_n^{O_n(r)}
\cdot\mathbf{y},\\
\mathcal{H}^-_X(\mathbf{x}) &=\sum_{\mathbf{y\in I}(g)}
\sum_{\{r\in \mathrm{Rect}^\circ(\mathbf{x,y})|r\cap\mathbb{X}=\{X_{j+1}\}\}}
U_{1}^{O_{1}(r)}\cdots U_n^{O_n(r)}
\cdot\mathbf{y}.
\end{align*}

\begin{lem}
\label{lem:parallel}
\begin{itemize}
    \item $\mathcal{H}^-_X$ is homogeneous of degree $(-1,-\omega(O_{n+1}))$ and $\mathcal{H}_O$ is homogeneous of degree $(1,\omega(X_{j+1}))=(1,\omega(O_{n+1}))$.
    \item $\mathcal{H}_X$ and $\mathcal{H}_O$ are chain homotopy equivalences.
\end{itemize}
\end{lem}
\begin{proof}
The proof is the same as Lemma 5.2.19,\cite{grid-book}.
Consider $\mathcal{H}^-_{O,X}\colon N\to N$ as
\[
\mathcal{H}^-_{O,X}(\mathbf{x})=\sum_{\mathbf{y\in N}(g)}
\sum_{\{r\in \mathrm{Rect}^\circ(\mathbf{x,y})|r\cap\mathbb{X}=\{X_{j+1}\},O_{n+1}\in r\}}
U_{1}^{O_{1}(r)}\cdots U_n^{O_n(r)}
\cdot\mathbf{y}.
\]
Then, it is sufficient to verify that $\mathcal{H}^-_X\circ\mathcal{H}^-_O=\mathrm{Id}_I$ and $\mathcal{H}^-_O\circ\mathcal{H}^-_X +\mathcal{H}^-_{O,X}\circ\partial_N^N+\partial_N^N\circ\mathcal{H}^-_{O,X}=\mathrm{Id}_N$.
The former equation is straightforward, and the latter equation holds by Lemma 5.2.19,\cite{grid-book}.
\end{proof}

We see the following commutative diagram:
\[
\xymatrix{
I\ar[r]^-{\partial_I^N}\ar[d]^-c & N\ar[d]^-{c\circ\mathcal{H}^-_{X}}\\
CF^-(g',\omega')[U_{n+1}]\llbracket 1,\omega(O_{n+1})\rrbracket\ar[r]^-{U} & CF^-(g',\omega')[U_{n+1}]
}
\]
where $U=U_{n+1}-U_n$ if there is just one $X$-marking in the column containing $X_{j+1}$ and $U=U_{n+1}$ otherwise.

\begin{lem}
A commutative diagram of bigraded $\mathcal{R}$-module chain complexes
\[
\xymatrix{
C\ar[r]^-f\ar[d]_-\phi & C' \ar[d]^-{\phi'} \\
D \ar[r]^-g & D'
}
\]
induces a bigraded chain map $\Phi\colon\mathrm{Cone}(f)\to\mathrm{Cone}(g)$.
If $\phi$ and $\phi'$ are quasi-isomorphisms, then $\Phi$ is a quasi-isomorphism.
\end{lem}

Using this lemma, we see that $H(\mathrm{Cone}(\partial_I^N))\cong\mathrm{Cone}(U)$.
Since $H(\mathrm{Cone}(\partial_I^N))\cong HF^-(g,\omega)$, it is sufficient to show that $\mathrm{Cone}(U)\cong HF^-(g',\omega')$.

By the definition of $CF^-(g',\omega')[U_{n+1}]$, we get $H(CF^-(g',\omega')[U_{n+1}])\cong H(CF^-(g',\omega'))[U_{n+1}]$.
Hence the injective map $U\colon CF^-(g',\omega')[U_{n+1}]\to CF^-(g',\omega')[U_{n+1}])$ induces the injective map $U\colon HF^-(g',\omega')[U_{n+1}]\to HF^-(g',\omega')[U_{n+1}])$ on homology.
Consider a long exact sequence of a mapping cone $\mathrm{Cone}(U)$, we get the following short exact sequence of absolute Maslov, relative Alexander graded $\mathbb{F}[U_1,\dots,U_{n+1}]$-modules:
\[
\xymatrix{
0\ar[r] & HF^-(g',\omega')[U_{n+1}]\ar[r]^-{U} & HF^-(g',\omega')[U_{n+1}]\ar[r] & H(\mathrm{Cone}(U))\ar[r] & 0 \\
}
\]
Therefore we see that $\frac{HF^-(g',\omega')[U_{n+1}]}{\mathrm{Im}U}\cong H(\mathrm{Cone}(U))$.
It is natural that $\frac{HF^-(g',\omega')[U_{n+1}]}{\mathrm{Im}U}\cong HF^-(g',\omega')$ as absolute Maslov, relative Alexander graded $\mathbb{F}[U_1,\dots,U_{n}]$-modules.
Combining these equations, we get $HF^-(g,\omega)\cong HF^-(g',\omega')$ as absolute Maslov, relative Alexander graded $\mathbb{F}[U_1,\dots,U_{n}]$-modules.
\end{proof}

\begin{rem}
If there is just one $X$-marking in the $column$ containing $A$, the same result is shown by a small modification.
Define $\mathcal{H}^-_X\colon N\to I$ as
\[
\mathcal{H}^-_X(\mathbf{x}) =\sum_{\mathbf{y\in I}(g)}
\sum_{\{r\in \mathrm{Rect}^\circ(\mathbf{x,y})|r\cap\mathbb{X}=\{X_{j}\}\}}
U_{1}^{O_{1}(r)}\cdots U_n^{O_n(r)}
\cdot\mathbf{y}.
\]
$\mathcal{H}^-_X$ is homogeneous of degree $(1,\omega(X_j))=(1,\omega(O_{n+1}))$.
We see that $\mathcal{H}^-_O$ and $\mathcal{H}^-_X$ are chain homotopy equivalences by using chain homotopy $\mathcal{H}^-_{O,X}$ defined as
\[
\mathcal{H}^-_{O,X}(\mathbf{x})=\sum_{\mathbf{y\in N}(g)}
\sum_{\{r\in \mathrm{Rect}^\circ(\mathbf{x,y})|r\cap\mathbb{X}=\{X_{j}\},O_{n+1}\in r\}}
U_{1}^{O_{1}(r)}\cdots U_n^{O_n(r)}
\cdot\mathbf{y}.
\]
Then, there is a commutative diagram
\[
\xymatrix{
I\ar[r]^-{\partial_I^N}\ar[d]^-c & N\ar[d]^-{c\circ\mathcal{H}^-_{X}}\\
CF^-(g',\omega')[U_{n+1}]\llbracket 1,\omega(O_{n+1})\rrbracket\ar[r]^-{U} & CF^-(g',\omega')[U_{n+1}]
}
\]
where $U=U_{n+1}-U_n$ if there is just one $X$-marking in the row containing $X_{j}$ and $U=U_{n+1}$ otherwise.
\end{rem}

\section{Applications}
\begin{lem}
\label{lem:ntimes}
Let $(f,\omega)$ be an MOY graph and $\omega_n$ be a balanced coloring for $f$ defined as $\omega_n(e)=n\cdot\omega(e)$ for any edge $e$.
Then there is a natural isomorphism $A_n\colon HF^-(f,\omega)\cong HF^-(f,\omega_n)$, where $A_n$ is a multiplication of Alexander grading by $n$.
\end{lem}
\begin{proof}
There is a natural isomorphism $a_n\colon CF^-(g,\omega)\to CF^-(g,\omega_n)$ defined by $\mathbf{x}\mapsto\mathbf{x}$.
Using \eqref{aa}, $a_n$ preserves Maslov grading and multiples Alexander grading by $n$.
\end{proof}

\begin{prop}
Let $(f,\omega)$ be an MOY graph and $(f^n,\omega^n)$ be the MOY graph obtained from $(f,\omega)$ by replacing each edge with $n$ parallel copies of the edge with the same orientation and weight as the original edge.
Then there is a natural isomorphism $A_n\colon HF^-(f,\omega)\cong HF^-(f^n,\omega^n)$, where $A_n$ is a multiplication of Alexander grading by $n$.
\end{prop}
\begin{proof}
Combine Theorem \ref{thm:main2} and Lemma \ref{lem:ntimes}.
\end{proof}
This Proposition is the partial extension of \cite[Theorem 4.5]{coloring}.

\section{Further problems}
Bao and Wu \cite{MOY-alexander} researched an Alexander polynomial for MOY graphs.
They defined Alexander polynomial using Kauffman state sum $\langle D,\omega\rangle$ for a given graph diagram $D$ of $f$ with balanced coloring $\omega$.
They summarized that the state sum $\langle D,\omega\rangle$ is an elaboration of the Alexander polynomial of Harvey and O'Donnol.
In their paper \cite{MOY-alexander}, Bao and Wu defined the Alexander polynomial as $\langle D,\omega\rangle/(t^{-1/2}-t^{1/2})^{V-1}$, where $V$ is the number of vertices of the graph.
Their Alexander polynomial of MOY graphs satisfies some relations that are proposed by Murakami-Ohtsuki-Yamada \cite{MOY-origin}.
So the natural question is whether these relations are shown in the framework of grid homology. 

Balanced spatial graphs researched by Vance\cite{grid-tau} and Kubota\cite{grid-upsilon-concordance} is a special case of MOY graphs.
They defined the tau invariant and the Upsilon invariant for balanced spatial graphs respectively.
By the definitions, it is straightforward that the tau invariant and the Upsilon invariant are extended to MOY graphs.
Then the question is what information about MOY graphs do these invariants contain?

\section{Acknowledgement}
I am grateful to my supervisor, Tetsuya Ito, for helpful discussions and corrections.
This work was supported by JST, the establishment of university fellowships towards
the creation of science technology innovation, Grant Number JPMJFS2123.

\bibliography{grid}
\bibliographystyle{amsplain} 

\end{document}